\documentclass{article}
\usepackage{amssymb,amsmath,amsfonts,amsthm,latexsym}

\newtheorem{theo}{Theorem}

\newtheorem{prop}[theo]{Proposition}
\newtheorem{lemm}[theo]{Lemma}
\newtheorem{conj}[theo]{Conjecture}

\theoremstyle{definition}
\newtheorem{rema}{Remark}
\newtheorem{defi}[rema]{Definition}

 \def\NN{{\mathbb N}}  
 \def\RR{{\mathbb R}}  
   
 \def\ZZ{{\mathbb Z}}

\def\La{\Lambda}

  \def\cG{{\cal G}}  
    
\def\cC{{\cal C}}   \def\cO{{\cal O}} \def\cU{{\cal U}}
   \def\cP{{\cal P}} 
    \def\cW{{\cal W}}
   \def\cR{{\cal R}}

\newcommand{\diff}{\operatorname{Diff}}

\title{
The $C^{1+\alpha }$ hypothesis in Pesin theory \emph{revisited}}

\author{Christian Bonatti, Sylvain Crovisier and Katsutoshi Shinohara}

\begin{document}

\maketitle

\begin{abstract}
We show that for every compact $3$-manifold $M$ there exists an open subset of $\diff ^1(M)$
in which every generic  diffeomorphism admits uncountably many ergodic probability measures 
which are hyperbolic while their supports are disjoint and
admit a basis of attracting neighborhoods and a basis of repelling neighborhoods.
As a consequence, the points in the support of these measures have no stable and no unstable manifolds.
This contrasts with the higher regularity case, 
where Pesin theory gives us the stable and the unstable manifolds 
with complementary dimensions at almost every point.
We also give such an example in dimension two, without local genericity.

{ \medskip

  }

\end{abstract}

\section{Introduction}

Pesin theory \cite{Pe} is a strong tool to study the 
hyperbolic behavior of non-uniformly hyperbolic 
systems. It describes the local dynamics along orbits which are individually hyperbolic
(in the sense that they have well defined Lyapunov exponents which are all non vanishing): 
for any point of such an orbit, the stable and unstable sets are immersed submanifolds with complementary dimensions.

The original proof was carried out under the assumption that 
the dynamics is of regularity $C^{1+\alpha}$. 
It is important to understand 
if such a regularity assumption is essential or not. 
Pugh (see \cite{Pu} \emph{The $C^{1+\alpha }$ hypothesis in Pesin theory}) gave an example of a $C^1$-diffeomorphism on a $4$-manifold having an 
orbit with non-zero well defined Lyapunov exponents but no invariant manifolds. 
Meanwhile, it is also known that Pesin theory can be valid in the $C^1$ setting 
under extra hypotheses, see for instance \cite{BaVa}.
More recently, \cite{ABC} proved that Pesin theory
works in the $C^1$ setting under the existence of a dominated splitting. 

It is interesting to know under what setting Pesin theory can be generalized to  $C^1$ dynamics.
For instance:
\begin{itemize}
 \item Pugh's counter-example gives an explicit orbit, but not a generic point of a hyperbolic measure\footnote{Pugh's example 
 is partially hyperbolic and all center Lyapunov exponents are  negative. As in  \cite{ABC}, the stable/unstable splitting is dominated.
 In Pugh's example the assumption on the Lyapunov exponents means that the norm of the center derivative of large iterates  
 decreases exponentially, whereas in \cite{ABC} the product of the norm of the center derivative along the orbit decreases exponentially; 
 for generic points of measures, these two conditions are equivalent.}. 
It is therefore natural to ask if such an example naturally appears as a regular point of some hyperbolic measures.
 \item Even if the answer of the first question is \emph{No} in full generality, it is natural to ask if  Pesin theory holds for 
 $C^1$-generic diffeomorphisms. 
\end{itemize}

 In  \cite{Pu} it is conjectured:
 \begin{conj}[Pugh]\label{c.pugh}
  If $Orb(p)$ is an orbit with well defined non-vanishing Lyapunov exponents of a $C^1$-diffeomorphism $ f\colon M\to M$,
where $M$ has dimension two,
and $\dim(E^s) = \dim(E^u)=1$, then Pesin's result holds: $W^s(p)$ is a $C^1$-curve
tangent at the point $p$ to $E^s$.

Indeed this might be true (on manifolds of  any dimension) whenever $E^s$ has dimension one.
\emph{Regularity}\footnote{By \emph{regularity} Pugh  means  that the Lyapunov exponent is given by 
the exponential rate of the product of the norm of 
the derivative (up to replace $f$ by a large finite iterate)}   is automatic on one-dimensional subspaces.
 \end{conj}

In this article, we show that there are some limitations 
for the $C^1$-Pesin theory, by giving a negative answer to both questions above, and to Pugh's Conjecture~\ref{c.pugh}. 
More precisely, we will show the following.

\begin{theo}\label{t.conceptual}
Let $M$ be a smooth compact manifold of dimension three.
We furnish $\diff^1(M)$ with the $C^1$-topology.
Then there exists a non-empty open set 
$\mathcal{U} \subset \diff^1(M)$
and a dense $\mathrm{G}_\delta$ subset $\mathcal{R} \subset
\mathcal{U}$ such that every $f \in \mathcal{R}$ admits
a hyperbolic ergodic probability measure $\mu$
such that every point in the support of $\mu$ has trivial stable and unstable sets.
\end{theo}

Our construction is based on bifurcations generated by non-dominated robust cycles
between periodic orbits whose differential satisfy some bounds.
We think that the phenomenon we give can be described in more general fashion. 
However, in this paper we concentrate on producing one specific example
so that the essence of the argument will be presented more clearly.

\begin{rema}\label{r.conceptual}
 In the open set $\cU$ that we will build, any $C^1$-generic diffeomorphism admits uncountably many
 hyperbolic ergodic measures $\mu$ with trivial stable an unstable sets, and with pairwise disjoint support which are Cantor sets.  
The dynamics on each of these support is a \emph{generalized adding machine}
(also called \emph{odometer} or \emph{solenoid},
see for example \cite{BS} or \cite{MM} for the definition) and therefore is uniquely ergodic.

Furthermore, each of these supports
admits a basis of attracting neighborhoods and a basis of repelling neighborhoods: 
they are \emph{chain-recurrent classes} of the dynamics without periodic point, 
which were called \emph{aperiodic classes} in~\cite{BC}. 
\end{rema}

The construction above can be generalized in any dimension $\geq 3$; in dimension $2$ there are no example of locally generic 
non hyperbolic diffeomorphism.  We therefore only give a non-generic example. 

\begin{theo}\label{t.pugh} Any smooth compact surface admits a $C^1$-diffeomorphism and a chain recurrent 
class $\cC$ which is conjugated to an adding machine, (hence uniquely ergodic) such that the measure supported on $\cC$ in hyperbolic. 
Furthermore, $\cC$ is the intersection of a nested sequence of successively 
attracting and repelling regions: in particular, the stable and unstable sets of any point $x\in \cC$ are equal to $\{x\}$.
\end{theo}

In the rest of this paper, we give the proof of Theorem~\ref{t.conceptual}. 
The statements of Remark~\ref{r.conceptual} will be proved inside 
the proof of Theorem~\ref{t.conceptual} except the
point that there appear uncountably many such ergodic measures,
which will be explained in Section~\ref{s.uncountable}.
Theorem~\ref{t.pugh} will be proved in Section~\ref{s.pugh}. 

\section{Hyperbolic adding machines with trivial stable/unstable sets}

Let us review some basic definitions and facts.
Let $f$ be a $C^1$-diffeomorphism on a compact smooth manifold $M$ of dimension $N$
with a Riemannian metric.
The \emph{stable set} $W^s(x)$ of  a point $x\in M$ is the set of the 
points whose orbit is asymptotic to the one of $x$, that is
\[
W^s(x) := \{ y \in M \mid d(f^n(x), f^n(y)) \to 0 \text{ as } n \to +\infty \},
\]
where $d$ denotes some distance function on $M$. 
The unstable set $W^u(x)$ is the stable set of $x$ for $f^{-1}$.
\smallskip

Let $\mu$ be an $f$-invariant ergodic probability measure on $M$.
Its support, denoted by $\mathrm{supp}(\mu)$, is the smallest closed set in $M$ which has full $\mu$-measure.
Oseledets' theorem provides us with the \emph{Lyapunov exponents} $\chi_1\geq \dots\geq \chi_N$ of $\mu$:
at $\mu$-a.e. point $x\in M$ and for every $i$ with $1\leq i\leq N$, 
the set
$$E_i =\{v\in T_xM \mid \text{if $v\neq 0$ then } \frac 1 n \log \|Df^n(x)(v)\|\to \chi_i \text{ as } n\to \pm \infty\},$$
forms a vector space whose dimension is equal to the multiplicity 
of $\chi_i$ in the sequence $\chi_1,\dots,\chi_N$. 
An $f$-invariant ergodic measure is called \emph{hyperbolic} if every Lyapunov exponent of it is non-zero.

Remember that the largest Lyapunov exponent $\chi_1$ also satisfies the following (see for instance~\cite{L});
\begin{equation}\label{e.limitexponent}\tag{L1}
\chi_1 =\lim_{n\to +\infty} \frac 1 n \int \log\|Df^n(x)\|d\mu(x).
\end{equation}
Note that by sub-additivity of the sequence $n\mapsto \int\log\|Df^n(x)\|d\mu(x)$, we have for each $n\geq 1$
\begin{equation}\label{e.inequalityexponent}\tag{L2}
\chi_1\leq \frac 1 n \int \log\|Df^n(x)\|d\mu(x).
\end{equation}
If we define the continuous map $Jf(x) := \log |\det(Df(x))|$ and set
$Jf(\mu) :=\int Jf(x)d\mu(x)$, then we have (again see~\cite{L}):
\begin{equation}\label{e.jacob}\tag{L3}
Jf(\mu)=\chi_1+\dots+\chi_N.
\end{equation}
In dimension $3$, we also denote by $\chi^-(\mu)\leq \chi^c(\mu)\leq \chi^+(\mu)$ the Lyapunov exponents of
the measure $\mu$. If $p$ is a periodic point we denote by
$\chi^-(p),\chi^c(p),\chi^+(p)$ the Lyapunov exponents
of the ergodic invariant probability measure supported evenly along the orbit of $p$ and 
by $Jf(\cO(p))$ their sum. In dimension $2$, we also use similar notations omitting $\chi^c$.
\medskip

Our construction is detailed by the following.
\begin{theo}\label{t.technical}
Given any compact $3$-manifold $M$, there is
a non-empty open set $\mathcal{U} \subset \diff^1(M)$
and a dense $\mathrm{G}_\delta$ subset $\mathcal{R} \subset \mathcal{U}$ such that for every $f \in \mathcal{R}$,
there exist two sequences of compact regions $(A_n)$, $(R_n)$ $(n \in \mathbb{N})$
and a sequence of hyperbolic periodic points $(p_n)$ satisfying the following properties:
\def\labelenumi  {(\ref{t.technical}-\theenumi)}
\begin{enumerate}
\item\label{i.t1} $R_{n+1}\subset A_{n+1}\subset R_n$ for 
every $n \in \mathbb{N}$.
\item\label{i.t2} $A_n$ is a disjoint union of $m_n$ disks, that is, 
$A_n = \coprod_{i=0}^{m_n-1} D_{n, i}$ where
$D_{n, i}$ are $C^1$-disks in $M$ such that $f(D_{n, i}) \subset \mathrm{Int}(D_{n, i+1})$ 
for every $i \in \ZZ/(m_n\ZZ)$
(\,$\mathrm{Int}(X)$ denotes the (topological) interior of $X$).
\item\label{i.t3} $R_n$ is a disjoint union of $m_n$ disks, that is,
$R_n = \coprod_{i=0}^{m_n-1} E_{n, i}$ where
$E_{n, i}$ are $C^1$-disks in $M$ such that  $f^{-1}(E_{n, i}) \subset \mathrm{Int}(E_{n, i+1})$ for every $i \in \ZZ/(m_n\ZZ)$.
\item\label{i.t4} $\max_{i \in \mathbb{Z}/(m_n\mathbb{Z})} \mathrm{diam}(D_{n, i})  \to 0$  (as $n \to +\infty$), where
$\mathrm{diam}(X)$ denotes the diameter of $X$.
\item\label{i.t5} $m_{n+1}>m_n$ for every $n \in \mathbb{N}$.
\item\label{i.t6} $Jf(x) < 1$ and $\log\|Df^{-1}(x)\|<2$ for every $x\in A_1$.
\item\label{i.t7} We have $p_n\in R_n$. The largest and smallest Lyapunov exponents satisfy
$\chi^+(p_n)>3$ and $\chi^-(p_n)<-1$.
\end{enumerate}
\end{theo}
Theorem~\ref{t.conceptual} follows immediately from Theorem~\ref{t.technical} and the following proposition:

\begin{prop}\label{p.intersection}
Consider a diffeomorphism and sequences $(R_n)$, $(A_n)$ and $(p_n)$
satisfying all the properties announced in Theorem~\ref{t.technical}. Then the intersection
$\cC=\bigcap_n A_n=\bigcap_n R_n$ is a Cantor set on which the restriction of $f$ is minimal and
uniquely ergodic (indeed it is a generalized adding machine).

The invariant probability measure supported on $\cC$ is hyperbolic, and 
every point in the support of it has trivial stable and unstable sets.
\end{prop}

We will use the following property which is a consequence of~\eqref{e.limitexponent}
and~\eqref{e.inequalityexponent}.

\begin{lemm}\label{l.limit}
Let $(\mu_n)$ be a sequence of ergodic measures 
which converges to an ergodic measure $\tilde{\mu}$
in the weak topology. Then we have 
$\limsup_{n}\chi^{+}(\mu_n) \leq \chi^{+}(\tilde{\mu})$.
\end{lemm}

\begin{proof}[Proof of Proposition~\ref{p.intersection}]
As in~\cite{universal}, the properties (2-\ref{i.t1}) to
(2-\ref{i.t5}) imply that
the invariant set $\cC=\bigcap A_n$ is a Cantor set and
the dynamics restricted to $\cC$ is conjugated to a (generalized) adding machine. 
In particular, there exists a unique invariant (ergodic) probability measure $\mu$
supported on $\cC$.

Since $\mathcal{C}$ is contained in the 
nested sequence of attracting regions $A_n$,
together with the fact that the diameter of each 
connected component of $A_n$ converges to $0$,
we deduce that each point 
of $\mathcal{C}$ has trivial unstable set.
By the same reasoning applied to $f^{-1}$ and $(R_n)$, 
we obtain the triviality of stable set of every point in  $\mathcal{C}$.

Let us consider the sequence of 
measures $(\delta_{\mathcal{O}(p_n)})_{n \in \mathbb{N}}$
(where $\delta_{\mathcal{O}(p_n)}$ denotes 
the ergodic invariant probability measure supported evenly along the 
orbit of $p_n$). 
We show that this sequence converges to $\mu$ in the weak topology.
Indeed, for every convergent subsequence, the support of the limit must be contained 
in $\mathcal{C}$. Since $\mathcal{C}$ is uniquely ergodic 
and the limit is an invariant measure, 
it must coincide with $\mu$. 
As a result, we see that the sequence  
$(\delta_{\mathcal{O}(p_n)})$ itself converges to $\mu$.

Then by Lemma~\ref{l.limit} and property (2-\ref{i.t7}), the extremal Lyapunov exponents are non-zero:
$$\chi^+(\mu)\geq \limsup_n\chi^+(p_n)\geq 3 \quad \text{and} \quad
\chi^-(\mu)\leq \liminf_n\chi^-(p_n)\leq -1.$$
By property~(2-\ref{i.t6}), we have $-2<\chi^-(\mu)$. By~\eqref{e.inequalityexponent}
and property~(2-\ref{i.t6}), we obtain $Jf(\mu)=\chi^-(\mu)+\chi^c(\mu)+\chi^+(\mu)<1$.
This implies $\chi^c(\mu)<0$, 
in particular $\chi^c(\mu)$ is 
non-zero, too.
Thus the measure $\mu$ is hyperbolic, which completes the proof.
\end{proof}

\section{A property $\cP$ on periodic points}
Let us recall some definitions (see also \cite{BCDG}). We fix a $C^1$-diffeomorphism $f$ of $M$ 
and two hyperbolic periodic points $p$ and $q$.
\smallskip

We say that $p$ is \emph{homoclinically related} to the orbit $\cO(q)$ of $q$
if the stable manifold $W^s(p)$ of $p$ has a transverse intersection point with the
unstable manifold $\cW^u(\cO(q))$ of $\cO(q)$ and also the unstable manifold $W^u(p)$ 
has a transverse intersection point with the stable manifold $\cW^s(\cO(q))$. 
If $p$ is homoclinically related to the orbit of $q$, then
the stable dimensions of $p$ and $q$ are equal (hence the unstable dimensions are also equal).

Suppose that the dimensions of the stable manifolds of $p$ and $q$ are different. 
We say that $p$ and $q$ \emph{belong to a
robust heterodimensional cycle} if there exists two transitive hyperbolic sets $K$ and $L$
containing $p$ and $q$ respectively, such that for any diffeomorphism $g$ that is $C^1$-close to $f$,
the intersections $W^s(L_g)\cap W^u(K_g)$ and $W^u(L_g)\cap W^s(K_g)$
between the stable and unstable sets of the continuations of $L$ and $K$ for $g$ are non-empty.
Having a robust heterodimensional cycle is a $C^1$-robust property (for the continuations of $p$ and $q$).
\smallskip

We say that $U\subset M$ is a \emph{filtrating set} if it is the intersection $U=A\cap B$ between two compact
sets $A, B \subset M$ such that $f(A)\subset \mathrm{Int}(A)$ and $f^{-1}(B)\subset \mathrm{Int}(B)$.
Note that if a hyperbolic periodic point $p$ belongs to a filtrating set $U$, then the whole orbit of $p$,
the periodic points homoclinically related to $p$
and the periodic points which belong to a robust heterodimensional cycle
associated with $p$, are also contained in $U$.
\smallskip

Consider a compact $f$-invariant set $\La$ and $k\geq 1$.
A $Df$-invariant splitting $T_{\La}M=E\oplus F$ into two non-trivial vector bundles $E,F$
over $\La$ is said to be {$k$-\emph{dominated}} if for
every $x\in \La$ and every pair of unit vectors $u\in E_x$ and
$v\in F_x$, the following inequality holds:
$$\|Df^k(x)(u)\| < \frac{1}{2}{\|Df^{k}(x) (v) \|}.$$
We say that $\La$ has no $k$-domination if there is no (non-trivial) $k$-dominated splitting $E\oplus F$ on $\La$.
\medskip

We introduce a property on hyperbolic periodic points.
\def\labelenumi{($\mathcal{P}$-\theenumi)}
\begin{defi}\label{d.p}
Let $f \in \mathrm{Diff}^1(M)$. A hyperbolic periodic point $p$ \emph{satisfies the property $\cP$} if it satisfies 
all the conditions below:
\begin{enumerate}
\item\label{i.p1} There is a hyperbolic periodic point $q_{1}$ whose stable eigenvalues are (not real and)
complex and which is  homoclinically related to the orbit of $p$.
\item\label{i.p2} There is a hyperbolic periodic point $q_{2}$ whose unstable eigenvalues are (not real and)
complex and which belongs to a robust heterodimensional cycle with the orbit of $p$.
\item \label{i.p3} There exist two hyperbolic periodic points $p^-,p^+$ homoclinically related to the orbit of $p$, such that
\begin{enumerate}
\item\label{i.p3a} $Jf(\cO(p^-))<0$ and $Jf(\cO(p^+))>0$,
\item\label{i.p3b} both two periodic points $p^{\pm}$ satisfy $\chi^+(p^\pm)>3$ and $\chi^-(p^\pm)<-1$.
\end{enumerate}
\end{enumerate}
\end{defi}

This property is clearly \emph{robust}: it is still satisfied for the hyperbolic continuation $p_g$ of
the diffeomorphisms $g$ that are $C^1$-close to $f$. The following states that it is non-empty.

\begin{prop}\label{p.cycle}
Any compact $3$-manifold $M$ admits a diffeomorphism $f$ having a filtrating set $V$ for $f$,
and a hyperbolic periodic point $p\in V$ which satisfies $\cP$.
Furthermore, such $f$ can be taken so that $Jf(x) < 1$ and $\log\|Df^{-1}(x)\|<2$ holds for every $x\in V$.
\end{prop}

The proof of Proposition~\ref{p.cycle} will be discussed later (see Section~\ref{p.conex}).
\medskip

We will prove that the property $\cP$ can be reproduced by perturbation 
at new periodic points with higher periods and separated from the initial periodic point $p$ by a filtrating set.
This idea was firstly used in \cite{universal} to build \emph{aperiodic classes}, 
using the lack of domination (properties ($\mathcal{P}$-\ref{i.p1}) and ($\mathcal{P}$-\ref{i.p2}))
and the existence of points homoclinically related to $p$ with Jacobian greater and less than $1$
(property ($\mathcal{P}$-\ref{i.p3a})), in order to prove the generic existence of \emph{universal dynamics}
(see \cite{universal} for detail).
Then \cite{survey} defined the notion of \emph{viral property}, 
which is an abstract formalization of the ``reproduction property" used in the proof of \cite{universal}:
viral properties always lead to the $C^1$-generic coexistence of uncountably many chain recurrent classes 
and conjectures in \cite{survey} propose that, conversely,  the $C^1$-locally generic 
coexistence of uncountably many  chain recurrent classes 
implies the existence of some viral property.
\begin{defi}
A property for hyperbolic periodic points is \emph{viral}
if it is $C^1$-robust and if, for any filtrating set $U$ containing $p$,
there is an arbitrarily $C^1$-small perturbation $g$ of $f$ which produces 
a periodic point $p'$ satisfying the property
and contained in a filtrating set $U'\subset U$ disjoint from $p_g$.
\end{defi}

In the next section we will prove the following theorem,
which essentially states that $\cP$ is viral.

\begin{theo}\label{t.viral}
Let $f \in \mathrm{Diff}^1(M)$ and $p$ be a hyperbolic periodic point satisfying the property $\cP$.
Then for every $C^1$-neighborhood $\cU$ of $f$, for every filtrating set $U$ containing $p$,
for every $\delta > 0$ and $m_0\geq 1$,
there exist $g\in \cU$, a hyperbolic periodic point $p'$ for $g$
and two compact regions $R\subset A\subset M$ satisfying the following:
\def\labelenumi  {(\ref{t.viral}-\theenumi)}
\begin{enumerate}
\item\label{i.v1} The periodic point $p'$ satisfies $\cP$ with period $m$ greater than $m_0$
and the whole orbit is contained in $R$.
\item\label{i.v2} $A$ is a disjoint union of $m$ disks $A = \coprod_{i=0}^{m-1} D_{i}$ such that 
$f(D_{i}) \subset \mathrm{Int}(D_{i+1})$ for each $i \in \ZZ/(m\ZZ)$
and $\max_{i \in \mathbb{Z}/(m\mathbb{Z})} \mathrm{diam}(D_{i})<\delta$.
\item\label{i.v3} $R$ is a disjoint union of $m$ disks 
$R = \coprod_{i=0}^{m-1} E_{i}$ such that $f^{-1}(E_{i}) \subset \mathrm{Int}(E_{i+1})$ for each $i \in \ZZ/(m\ZZ)$
and $\max_{i \in \mathbb{Z}/(m\mathbb{Z})} \mathrm{diam}(E_{i})<\delta$.
\end{enumerate}
\end{theo}

We now give the proof of Theorem~\ref{t.technical} by a genericity argument.

\begin{proof}[Proof of Theorem~\ref{t.technical} from Proposition~\ref{p.cycle} and Theorem~\ref{t.viral}]
By Proposition~\ref{p.cycle}, there exist a diffeomorphism $f_0$, a filtrating set $U$,
a $C^1$-neighborhood $\cU$ of $f_0$ such that any diffeomorphism $f\in \cU$ satisfies the property~(\ref{t.technical}-\ref{i.t6})
of Theorem~\ref{t.technical} for every point $x\in U$ and 
admits a hyperbolic periodic point $p\in U$ with the property $\cP$.

Let $\delta_n>0$ be a sequence tending to $0$ as $n\to +\infty$. By Theorem~\ref{t.viral}, we inductively build a sequence of $C^1$-open sets
$\cU\supset \cG_1\supset \cG_2\supset\cdots$ such that every $\cG_n$ is dense in $\cU$ 
and for every $h\in \cG_n$, we can find some 4-ple $(A_n,R_n, m_n, p_n)$ associated to $\delta_n$, 
where $R_n \subset A_n\subset U$ are two compact subsets, $m_n$ is an integer,
and $p_n$ is a hyperbolic periodic point satisfying the properties of the Theorem~\ref{t.technical}.
We can assume that $A_n,R_n,m_n$ are locally constant on $\cG_n$
and that $p_n$ depends continuously on $f$.  
Then the $\mathrm{G}_{\delta}$ subset $\cR:=\cap \cG_i$ of $\cU$ is dense by Baire's category theorem and 
satisfies the conclusion of Theorem~\ref{t.technical}.
\end{proof}

In order to prove Proposition~\ref{p.cycle},
we will use a different version of Theorem~\ref{t.viral},
where property $\cP$ in the assumption is replaced 
by a slightly different property $\cP'$ (which is not robust).

\begin{defi}\label{d.pp}
Let $f \in \mathrm{Diff}^1(M)$. A hyperbolic periodic point $p$ and a homoclinic point
$z\in W^s(p)\cap W^u(p)$ \emph{satisfy the property $\cP'$} if:
\def\labelenumi{($\mathcal{P}'$-\theenumi)}
\begin{enumerate}
\item\label{i.pp1} $\chi^-(p)< -1$, $\chi^+(p)>3$ and $Jf(\cO(p))=0$.
\item\label{i.pp2} The union $\cO(p)\cup\cO(z)$ of the orbits of $p$ and $z$
does not admit any dominated splitting.
\end{enumerate}
\end{defi}
Note that the point $z$ in the previous definition must be a homoclinic tangency of the orbit of $p$:
indeed, if the intersection $W^s(p)\cap W^u(p)$ is transverse at $z$, then Smale's intersection theorem
implies that $\cO(p)\cup\cO(z)$ is a hyperbolic set, which contradicts 
the lack of domination~($\mathcal{P}'$-\ref{i.pp2}).

\begin{theo}\label{t.viral2}
Let $f \in \mathrm{Diff}^1(M)$, $p$ a hyperbolic periodic point
and $z$ a homoclinic point of $p$ satisfying the property $\cP'$.
Then for every $C^1$-neighborhood $\cU$ of $f$ and for every
neighborhood $U$ of $\cO(p)\cup\cO(z)$,
there exist $g\in \cU$, a filtrating set $V\subset U$ and
a hyperbolic periodic point $p'\in V$ which satisfies the property $\cP$.
\end{theo}

\section{Virality of the property $\cP$ (Theorems~\ref{t.viral} and~\ref{t.viral2})}
\label{ss.viral}

The proof of Theorems~\ref{t.viral} and~\ref{t.viral2}
is a modification of the argument of \cite[Proposition 9.4]{BCDG},
where the virality of a property $\mathfrak{V}''$ was proved. 
For $3$-dimensional manifolds, the definition of $\mathfrak{V}''$ can be stated as follows.

\begin{defi}
A hyperbolic periodic point $p$ of $f$ \emph{satisfies the
property $\mathfrak{V}''$} if conditions~($\mathcal{P}$-\ref{i.p1}) and ($\mathcal{P}$-\ref{i.p2}) of 
Definition~\ref{d.p} hold.
\end{defi}

The proof of the virality of the property $\mathfrak{V}''$ in~\cite{BCDG} consists of six steps.
Let us fix a diffeomorphism $f$, a $C^1$-neighborhood $\cU$ of $f$,
a filtrating set $U$, a hyperbolic periodic point $p\in U$ satisfying $\mathfrak{V}''$,
an integer $m_0$ and $\delta>0$.
Then there are positive integers $k$ and $\ell$ 
such that any periodic saddle of $f$ (or of a $C^1$-perturbation of $f$)
which has no $k$-domination and period larger than $\ell$ may be turned into a sink or a source
(see~\cite[Lemma 4.3]{BCDG}) and may give birth to a homoclinic tangency (see~\cite[Lemma 2.1]{BCDG})
by a $C^1$-perturbation in $\cU$.

\begin{enumerate}
\item[]  \emph{Step I. Selection of the saddle $p'$.}
This step consists of selecting (after an arbitrarily $C^1$-small perturbation $f_1$ of $f$)
a periodic point $p'$ of the diffeomorphism $f_1$ homoclinically related to the orbit of
$p$ which has no $k$-dominated splitting and whose period $m$ is larger than $\ell$ and $m_0$.
Indeed since  the orbits of $p$ and $q_2$ belong to a robust heterodimensional cycle,
for any $C^1$-generic diffeomorphism $f_1$ which is $C^1$-close to $f$,
there exists a transitive locally maximal hyperbolic set $\Lambda$
which contains $p$, $q_1$, $p^-$, $p^+$
and a point arbitrarily close to $q_2$.
From properties ($\mathcal{P}$-\ref{i.p1}) and ($\mathcal{P}$-\ref{i.p2}),
the set $\Lambda$ has no $k$-dominated splitting.
One then chooses a periodic point $p'\in \Lambda$ whose orbit is sufficiently close to
$\Lambda$ in the Hausdorff topology.

Since $U$ is filtrating, $\cO(p')$ is contained in $U$.

\item[] \emph{Step II. Separation of the saddle.}
The lack of domination along the periodic orbit of $p'$
allows us to turn it into a sink or a source, depending on the sign of $Jf_1(\cO(p'))$.
Hence there exists a perturbation of $f_1$ in $\cU$ with small support
which creates a compact set $A$ or $R$ contained in $U$, containing $p'$, disjoint from the orbit of $p$
and satisfying property (\ref{t.viral}-\ref{i.v2}) or (\ref{t.viral}-\ref{i.v3}) of Theorem~\ref{t.viral}.
By a new perturbation (we denote the resulted diffeomorphism by $f_2$), 
one can ``recover" the original diffeomorphism $f_1$ in a smaller
neighborhood of $\cO(p')$. In particular, the differential of $f_2$ is equal to
that of $f_1$ along $\cO(p')$ and $f_2$ still has no $k$-dominated splitting
along $\cO(p')$.
\item[] \emph{Step III. New periodic orbits homoclinically related to $p'$.}
The lack of domination along the periodic 
orbit $\cO(p')$ allows us to create a horseshoe by unfolding a homoclinic tangency associated to $p'$,
using the result of~\cite{Go}. Again, we require that for this new perturbation $f_3$ the tangent maps
$Df_3$, $Df_2$ (and $Df_1$) coincide along the orbit of $p'$ (see also Remark~\ref{r.gou} below).
The perturbation is supported in an arbitrarily small neighborhood of $\cO(p')$.
\item[] \emph{Steps IV, V, VI.} After a new perturbation $g\in \cU$ of $f_3$, 
we turn the periodic point $p'$
satisfying the property $\mathfrak{V}''$. The perturbation is realized in an arbitrarily
small neighborhood of any periodic orbit homoclinically related to $p'$.
\end{enumerate}
\begin{rema}\label{r.gou}
In Step III, the last requirement was not justified in~\cite{BCDG}, but it is a consequence
of the results of~\cite{Go,Go2}.
Indeed, the homoclinic bifurcation is obtained from the lack of $k$-domination along the periodic orbit
$\cO(p')$ by applying~\cite[Theorem 3.1]{Go}: in~\cite[section 6.1]{Go2}, it is
proved that this perturbation can be performed by preserving the periodic orbit $\cO(p')$ and the
derivatives along it (see~\cite[Theorem 8]{Go2}).
\end{rema}

This gives a version of Theorem~\ref{t.viral} for property $\mathfrak{V}''$,  
where only one of the conditions~(\ref{t.viral}-\ref{i.v2}) or~(\ref{t.viral}-\ref{i.v3}) can be obtained
depending on the sign of $Jf_1(\cO(p'))$. 
We now explain how to modify the previous argument.
The modification of steps I and II allows us to build both attracting and repelling regions $A$ and $R$
using  either the condition~($\mathcal{P}$-\ref{i.p3a}) of Definition~\ref{d.p} 
or the condition~($\mathcal{P}'$-\ref{i.pp2}) of Definition~\ref{d.pp}.
The modification of steps I and III allows us to obtain
the condition~($\mathcal{P}$-\ref{i.p3}) for the new point $p'$.
The proofs of Theorems~\ref{t.viral} and~\ref{t.viral2} 
are very similar and only differ in their first step.
So we present them simultaneously except the first step.

\medskip

\begin{proof}[Proof of Theorems~\ref{t.viral} and~\ref{t.viral2}]

Let $f$ be a $C^1$-diffeomorphism satisfying the assumption of Theorem~\ref{t.viral}
or Theorem~\ref{t.viral2}.
Let us fix a $C^1$-neighborhood $\cU$ of $f\in \diff^1(M)$,
a filtrating set $U$, a hyperbolic periodic orbit $\cO(p)\subset U$, an integer $m_0 >0$ and $\delta>0$.
As before, we fix positive integers 
$k$ and $\ell$ such that any periodic saddle of $f$ (or of a $C^1$-perturbation of $f$)
which has no $k$-domination and period larger than $\ell$ may be turned into a sink or a source
and may give birth to a homoclinic tangency by a $C^1$-perturbation in $\cU$.
\bigskip

\noindent\textbf{Step I.}
As is in \cite{BCDG}, after some 
preliminary perturbation we select a periodic point which 
does not admit $k$-dominated splitting along the orbit with certain conditions.

\medskip

\noindent \emph{Proof under the assumptions of Theorem~\ref{t.viral}.}
Arguing as in the proof of the step I in~\cite{BCDG},
by the conditions~($\mathcal{P}$-\ref{i.p1}) and ($\mathcal{P}$-\ref{i.p2}),
there exists a $C^1$-small perturbation $f_1$ of $f$
and a transitive hyperbolic set $K$ that is locally
maximal, contains $p,p^-,p^+$ and  
has no $k$-dominated splitting.

We then use the condition~($\mathcal{P}$-\ref{i.p3}).
A $C^1$-small perturbation ensures that $f$ belongs to a dense $\mathrm{G}_\delta$
subset of $\diff^1(M)$ such that the properties of~\cite[Lemma 4.1]{BCDG} and~\cite[Corollary 2]{ABCDW}
hold. In particular, \cite{ABCDW} implies that there exists a periodic point $p'\in K$
whose Lyapunov exponents $(\chi^-(p'),\chi^c(p'),\chi^+(p'))$ are arbitrarily close to the barycenter
$$\frac{J^+}{J^-+J^+}\cdot(\chi^-(p^-),\chi^c(p^-),\chi^+(p^-))\;+
\;\frac{J^-}{J^-+J^+}\cdot(\chi^-(p^+),\chi^c(p^+),\chi^+(p^+)),$$
where $J^+=|Jf(\cO(p^+))|$ and $J^-=|Jf(\cO(p^-))|$.
Note that since $Jf(\cO(p^+))>0$ and $Jf(\cO(p^-))<0$, this value is equal to $0$.
Thus the quantity $Jf(\cO(p'))$
can be taken arbitrarily close to $0$. 
Then by a small $C^1$-perturbation $f_1$
(given by Franks' lemma, see~\cite{Go2})
whose effect on the derivative of $f$ along the orbit of $p'$ is a multiplication
by a homothety close to identity,
one can ensure that $Jf_1(\cO(p'))=0$, keeping the Lyapunov exponents almost unchanged. Furthermore,  
since the perturbation is arbitrarily $C^1$-small, the transitive hyperbolic set $K$ 
has a hyperbolic continuation.  Therefore,
the points $p, p'$ still belong to a same transitive hyperbolic set and are homoclinically related.
By construction, $p'$ and $f_1$ satisfies ($\mathcal{P}'$-\ref{i.pp1}) in Definition~\ref{d.pp}.

We complete this step as before:
from \cite[Lemma 4.1]{BCDG}, the periodic orbit $p'$ above can be chosen 
arbitrarily close to
$K$ with respect to the Hausdorff topology,
hence there is no $k$-domination along the orbit of $p'$ for $f_1$, and the period may be chosen
larger than $\ell$ and $m_0$. Since $U$ is filtrating and $p$ belongs to $U$,
the orbit of $p'$ is contained in $U$.
\medskip

\noindent \emph{Proof under the assumptions of Theorem~\ref{t.viral2}.}
By unfolding the homoclinic tangency at $z$, 
we create a horseshoe $\Lambda\subset U$ containing $p$
and close to $C=\cO(p)\cup\cO(z)$ in the Hausdorff topology, by an arbitrarily
$C^1$-small perturbation $f_1$ of $f$. 

We can assume there is no $k$-dominated splitting
on $\Lambda$:
indeed,  as $f_1$ converges to $f$, the horseshoe 
$\Lambda = \Lambda_{f_1}$ converges to $C$.
If every $\Lambda_{f_1}$ admits $k$-dominated splitting,
then, since dominated splitting is preserved
by taking limits (see~\cite[Appendix B.1]{BDV}),
it means that $C$ admits a dominated splitting,
but it contradicts to the assumption.
We furthermore assume that $f_1$ satisfies the conclusion of~\cite[Lemma 4.1]{BCDG}.

Now we see that there exists a periodic point $p'\in \Lambda$
close to $\Lambda$ in the Hausdorff topology, whose Lyapunov exponents are close to the exponents of $p$.
Then by an arbitrarily $C^1$-small perturbation $f_1$ given by Franks' lemma,
the Lyapunov exponents of $p'$ for $f_1$ coincide with those of $p$ for $f$.
As before, since $\cO(p')$ and $\Lambda$ are close,
there is no $k$-domination along the orbit of $p'$ and the period may be chosen larger than $\ell$ and $m_0$.
\bigskip

\noindent
\textbf{Step II.} We repeat the step II of~\cite{BCDG}.
Under our current assumption $Jf_1(\cO(p'))=0$, 
we know that the orbit of $p'$ can be 
turned both to a sink and to a source, 
which gives both of the regions $A$ and $R$ containing $\cO(p')$.
This gives the filtrating region $V$ in Theorem~\ref{t.viral2}. 
For the case of Theorem~\ref{t.viral},
since the perturbation can be performed locally, 
we can assume that $A$ and $R$ are sufficiently small so that 
the conditions~(\ref{t.viral}-\ref{i.v2}) and~(\ref{t.viral}-\ref{i.v3})
holds for the constant $\delta$.
\bigskip

\noindent\textbf{Step III.}
We first repeat the step III of~\cite{BCDG}:
according to Remark~\ref{r.gou}, the lack of domination along the orbit of 
$p'$ allows us to create a homoclinic tangency. 
Then by unfolding it we create a non trivial horseshoe 
associated $p'$ without changing the derivative along the orbit of $p'$.
The perturbation is supported in a small neighborhood of $\cO(p')$,
hence we can asuume that the sets $A$, $R$ or $V$ satisfy the desired property.

Then, after an arbitrarily $C^1$-small perturbation 
we create periodic orbits with arbitrarily large period
that are homoclinically related 
to $p'$ with Lyapunov exponents being close
to those of $p'$. Indeed, one can perturb the diffeomorphism
so that it belongs to the dense $\mathrm{G}_\delta$-set of diffeomorphisms satisfying
the conclusion of~\cite[Lemma 4.1]{BCDG}.
This enables us to obtain two hyperbolic periodic points ${p'}^-,{p'}^+$
whose orbits are homoclinically related to $p'$, 
satisfying the property~($\mathcal{P}$-\ref{i.p3b}) of Definition~\ref{d.p}
and such that $Jf_3(\mathcal{O}({p'}^\pm))$ are close to $0$. As before, by perturbation of $f_3$ given by Franks' lemma,
we can ensure the property~($\mathcal{P}$-\ref{i.p3a}), that is, $Jf_3(\cO({p'}^-))<0$ and $Jf_3(\cO({p'}^+))>0$.
\bigskip

\noindent
\textbf{Final Step.} We repeat the steps IV, V, VI of~\cite{BCDG}
in order to build a last perturbation $g\in \cU$ such that
conditions~($\mathcal{P}$-\ref{i.p1}) and~($\mathcal{P}$-\ref{i.p2}) of Definition~\ref{d.p} hold.
They can be performed outside a neighborhood of
a transitive hyperbolic set which contains $p',{p'}^-$ and ${p'}^+$. In particular ${p'}^-,{p'}^+$
still satisfy the property~($\mathcal{P}$-\ref{i.p3}) of Definition~\ref{d.p} and 
thus $p'$ has the property $\cP$.
\medskip

This concludes the proof of Theorems~\ref{t.viral} and ~\ref{t.viral2}.
\end{proof}

\section{Construction of a diffeomorphism satisfying $\cP$}
\label{p.conex}

In this section we give the proof of Proposition~{\ref{p.cycle}}.

\begin{proof}[Proof of Proposition~{\ref{p.cycle}}]
By deforming a linear automorphism
one can easily build a diffeomorphism $F$ of $\RR^3$ such that:
\begin{itemize}
\item For every $x\in \RR^3\setminus B(\boldsymbol{0},1)$, one has $F(x)=x$,
where $B(x, r)$ denotes the three dimensional ball centered at $x$ with radius $r$ 
and $\boldsymbol{0}$ is the origin of $\RR^3$.
\item In a neighborhood of the origin $\boldsymbol{0}$, the diffeomorphism $F$ has the form
\[ 
(r,s,t)\mapsto (\exp(-8/5)r, \, \exp(-8/5)s, \, \exp(16/5)t).
\]
In particular, $\boldsymbol{0}$ is the hyperbolic fixed point of $F$.
\item There exists a point of homoclinic tangency $a$ between $W^s(\boldsymbol{0})$ and $W^u(\boldsymbol{0})$.
\item For every $n \in \mathbb{Z}$ one has $ JF(F^n(a))<1$ and $\log\|DF^{-1}(F^N(a))\|<2$.
In particular, there exists an open neighborhood $W$ of the union $\cO(\boldsymbol{0})\cup\cO(a)$
such that $JF<1$ and $\log\|DF^{-1}\| <2$ on $W$.
\end{itemize}

\medskip

Let us consider any closed $3$-dimensional manifold $M$.
Since $F$ coincides with the identity outside the ball $B(\boldsymbol{0},1)$,
it can be realized as the restriction of a diffeomorphism of $M$:
there exist a map $h\colon B(\boldsymbol{0},1)\to M$ which is a diffeomorphism to its image $B=h(B(\boldsymbol{0},1))$
and a diffeomorphism $f$ of $M$ so that the restriction of $f$ to $B$ coincides with $h F h^{-1}$.

Let us define $p=h(\boldsymbol{0})$, $z=h(a)$ and $U=h(W)$.
Since the two stable eigenvalues at $p$ coincide, there is no dominated splitting $E\oplus F$
above $\cO(p)\cup\cO(z)$ such that $E$ is one-dimensional.
Furthermore, since $z$ is a homoclinic tangency,
there is no dominated splitting such that $E$ is two-dimensional.
The points $p$ and $z$ thus satisfy the property $\cP'$ and
the assumptions of Theorem~\ref{t.viral2} hold for $f$. 
Hence the Proposition~\ref{p.cycle} follows.
\end{proof}

\section{Uncountability of sets supporting hyperbolic measures with trivial (un)stable sets }\label{s.uncountable}
As explained in \cite{survey}, viral properties always lead to the 
generic coexistence of uncountably many aperiodic classes.  
Since the argument is very short, we recall it here.  This concludes the 
uncountability of the ergodic measures in Remark~\ref{r.conceptual}.

By repeating the proof of Theorem~\ref{t.technical} inductively, 
for each diffeomorphism $f$ in a dense $\mathrm{G}_\delta$ subset of $\cU$, 
each $n\geq 1$ and each word $w\in \{0,1\}^n$,
we can obtain compact sets $R_{w}\subset A_{w}$, an integer $m_{w}$
and a hyperbolic periodic point $p_{w}$ such that properties~(\ref{t.technical}-\ref{i.t2}), (\ref{t.technical}-\ref{i.t3}), 
(\ref{t.technical}-\ref{i.t4}) and (\ref{t.technical}-\ref{i.t7}) in Theorem~\ref{t.technical} are satisfied.
Furthermore, we can construct them so that $A_{w'}\subset R_{w}$ and $m_{w'}>m_{w}$ holds
when the first $n$ symbols of $w'\in \{0,1\}^{n+1}$ coincide with $w$, which
corresponds to the conditions (\ref{t.technical}-\ref{i.t1}) and (\ref{t.technical}-\ref{i.t5}) respectively.
One can also require that the sets $A_w$ for all $w\in \{0,1\}^n$ corresponding to a fixed integer $n$
are pairwise disjoint. For each $w\in \{0,1\}^\NN$, we denote by $w_n$ the sequence of first $n$
symbols of $w$. Then for each $\omega$ the sequence $(A_{w_n},R_{w_n},p_{w_n})$ satisfies
the hypotheses of Theorem~\ref{t.technical}
(note that the condition (\ref{t.viral}-\ref{i.t6}) can be established easily).
Hence, the intersections $\cC_w=\bigcap A_{w_n}$ for different sequences $w$ are pairwise disjoint aperiodic classes 
satisfying the conclusion of Proposition~\ref{p.intersection}.
Since there are uncountably many such sequences, Remark~\ref{r.conceptual} follows.

\section{A counterexample in dimension $2$}\label{s.pugh}

In this section, we prove Theorem~\ref{t.pugh}:  we give an example of a diffeomorphism of surface 
with a hyperbolic measure such that each point in the support has trivial stable and unstable sets. 

We start from a diffeomorphism $f$ on a compact two dimensional manifold $M$ such that: 
\begin{itemize}
\item[(H-1)] There exists a hyperbolic periodic saddle $p$ with $J(\cO(p))=0$, $\chi^{+}(p) >1$ and $\chi^{-}(p) <-1$.
\item[(H-2)] There exists $x \in W^u(p) \cap W^s(p)$ at which $T_xW^u(p) = T_xW^s(p)$.
\end{itemize}
It is not difficult to construct such diffeomorphism on any surface.

For such $f$, we prove the following:
\begin{prop}\label{p.pugh}
Suppose $f$ satisfies (H-1) and (H-2). Then, for any $C^1$-neighborhood $\cU$ of $f$,
for any neighborhood $U$ of $\cO(p)\cup \cO(x)$, 
for every $\delta>0$
and $m_0\geq 1$, there exists $g\in \cU$, points $p',x'$ and two compact regions $R\subset A \subset U$
such that the following holds.
\begin{itemize}

\item The points $p', x'\in \mathrm{Int}(R)$ satisfy the conditions (H-1) and (H-2) and $p$ has period $m\geq m_0$;

\item $A$ is a disjoint union $A = \coprod D_i$ of $m$ disks of diameter smaller than $\delta$ such that 
$g(D_i) \subset \mathrm{Int}(D_{i+1})$ (that is, $A$ is an attracting region). 

\item $R$ is a disjoint union $R = \coprod E_i$ of $m$ disks such that 
$g^{-1}(E_i) \subset \mathrm{Int}(E_{i-1})$ (that is, $R$ is a repelling region).
\end{itemize}
\end{prop}

Using Proposition~\ref{p.pugh} repeatedly, we can build  a sequence of diffeomorphisms  $C^1$-converging to a diffeomorphism $f_\infty$ 
presenting, as in Theorem~\ref{t.technical}, a nested sequence of periodic attracting/repelling
small disks of period tending to infinity,  containing periodic points 
with Lyapunov exponents greater than $1$ and less than $-1$. As in Proposition~\ref{p.intersection}, the limit is a chain recurrent 
class $\cC$ which is an adding machine, and the semi-continuity of the extremal Lyapunov exponents implies that the unique invariant measure 
is hyperbolic  (with exponents greater than $1$ and less than $-1$); finally
the stable/unstable sets of any point in $\cC$ are trivial.

\begin{proof}[Proof of Proposition~\ref{p.pugh}]
It is essentially the same as the proof of Theorem~\ref{t.viral}.
\medskip

\noindent
{\bf Step I.}  First, we unfold the homoclinic tangency at $x$ .
It produces a hyperbolic basic set (a horseshoe) $K\subset U$ containing $p$ and having a point arbitrarily close to $x$.
As a consequence, there is a hyperbolic periodic point $p'\in K$ with arbitrarily large period,
$J(\cO(p'))$ arbitrarily close to $0$, 
and Lyapunov exponents  
$\chi^+(p')> 1$ and $\chi^-(p')<-1$ and such that the hyperbolic splitting at $p'$ has an arbitrarily small angle: in particular, 
the dominated splitting is arbitrarily weak on $\cO(p')$.

\medskip

\noindent
{\bf Step II.} Using the absence of $k$-domination along $p'$,
and the fact that $J(\cO(p'))$ almost vanishes, we can construct 
the repelling and the attracting region around $p'$ keeping the local dynamics around
$\mathcal{O}(p')$ unchanged (on surfaces the argument goes back to Ma\~n\'e~\cite{mane}).
Note that the size of the regions can be taken arbitrarily small.

\medskip

\noindent
{\bf Step III.} Again, because of the absence of the domination, we can produce a 
point $x'$ of homoclinic tangency associated to $p'$ recovering hypothesis (H-2),
using~\cite{Go}.
A final pertubation preserving the tangency allows to get
 $J(\cO(p'))=0$, keeping the  bounds  $\chi^+> 1$ and $\chi^-<-1$.
We have recovered (H-1), ending the proof.
\end{proof}

\bigskip
\footnotesize
\noindent\textit{Acknowledgments.}
This work was partially supported by 
the Aihara Project (the FIRST program
from JSPS, initiated by CSTP),
the ANR project \emph{DynNonHyp} BLAN08-2 313375
and by the Balzan Research Project of J. Palis.

\vspace{1.5cm}

\begin{itemize}
\item[]  \emph{Christian Bonatti}
\begin{itemize}
\item[] Institut de Math\'ematiques de Bourgogne CNRS - URM 5584
\item[] Universit\'e de Bourgogne Dijon 21004, France
\end{itemize}
\item[] \emph{Sylvain Crovisier}
\begin{itemize}
\item[]  Laboratoire de Math\'ematiques d'Orsay CNRS - UMR 8628
\item[]  Universit\'e Paris-Sud 11  Orsay 91405, France
\end{itemize}
\item[] \emph{Katsutoshi Shinohara}
\begin{itemize}
\item[] FIRST, Aihara Innovative Mathematical Modelling Project, JST,
\item[] Institute of Industrial Science, University of Tokyo, 
\item[] 4-6-1, Komaba, Meguro-ku, Tokyo 153-8505, Japan
\end{itemize}
\end{itemize}

\end{document}